\documentclass[12pt]{amsart}%{article}

\usepackage{amssymb,enumerate}
\usepackage{url} 
\newtheorem{statement}{}
\newtheorem{theorem}[statement]{Theorem}

\newtheorem{proposition}[statement]{Proposition}

\newtheorem{corollary}[statement]{Corollary}

\newtheorem{rem}[statement]{Remark}

\def\N{\Bbb N}

\def\dens{\, {\rm dens}\, }

\def\la{\langle}
\def\ra{\rangle}

\def\vv{{\cal V}}

\def\ee{\varepsilon}

\def\nn{\|\cdot\|}
\def\ll{{\mathcal L}}
\def\vv{{\mathcal V}}
\def\lll{\longmapsto}

\def\C{\mathcal{C}}
\def\ov{\overline}

\title{Projections in duals to Asplund spaces made without Simons' lemma}

\author{Marek C\'uth and Mari\'an Fabian}

\dedicatory{In honor of the 70th birthday of Charles Stegall.}

\address{Faculty of Mathematics and Physics, Sokolovsk\'a 83, 18675 Praha 8, Czech Republic}
\address{Mathematical Institute of Czech Academy of Sciences, \v Zitn\'a 25, 115 67 Praha 1, Czech Republic}

\email{cuthm5am@karlin.mff.cuni.cz}
\email{fabian@math.cas.cz}

\thanks{M. C\'uth was supported by the Grant No. 282511/B-MAT/MFF of 
Grant Agency of Charles University in Prague. M. Fabian was supported by grant P201/12/0290 and by RVO: 67985840}

\date{\today}



\begin{document}\maketitle

\begin{abstract}
G. Godefroy and the second author of this note proved in 1988 
that in duals to Asplund spaces there always exists a projectional resolution of the identity. 
A few years later, Ch. Stegall succeeded to drop from the original proof a deep lemma of S. Simons. 
Here, we rewrite the condensed argument of Ch. Stegall in a more transparent and detailed way. 
We actually show that this technology of Ch. Stegall leads to a bit 
stronger/richer object ---the so called projectional skeleton--- recently constructed by W. Kubi\'s, via S. Simons' lemma and with help of elementary submodels from logic.
%In this note, we want to provide a detailed proof of this fact and thus attract a broader audience to this
%rather surprising simplification.
\end{abstract}

\maketitle

%\section{Introduction}

In 1988, G. Godefroy and the second named author of this note constructed in \cite{fg} a 
projectional resolution of the identity in duals to general Asplund spaces;
see, e.g., \cite[Definitions 1.0.1 and 6.1.5]{f}. A few years later, Ch. Stegall, presented a simplified variant of this construction by avoiding from the proof the use of S. Simons' lemma
 \cite[Lemma 8.1.3]{f} or of any other substitute of it. He published it, in a rather condensed form in \cite{st}. In this note, we perform his argument in full details. Thus we believe that Ch. Stegall's approach, so far overlooked by audience, will attract a broader attention. 
Ch. Stegall's technology led us even to constructing a stronger/richer object, which includes the so called 1-{\em projectional skeleton}, introduced and studied recently by W. Kubi\'s, see \cite[pages 765, 766]{kubis}, \cite[pages 369, 370]{kkl}. W. Kubi\'s' construction was based on S. Simons' lemma and
was done by using elementary submodels from logic. 
%\section{Construction of a P.R.I. on $X^*$}

Let $(X,\|\cdot\|)$ be any Banach space.  If $V$ is a subspace of $X$ and $x^*
\in X^*$, then ${x^*}\!\!\upharpoonright_{V}$ means the restriction of $x^*$ to $V$; similarly, we put
$M\!\!\upharpoonright_{V}:= \big\{{x^*}\!\!\upharpoonright_{V}:\ x^*\in M\}$ for any set $M$ in $X^*$.
Alaoglu's theorem asserts that the closed unit ball $B_{X^*}$ in $X^*$ provided with the weak$^*$ topology is a compact space.
Let ${\C(B_{X^*})}$ denote the (Banach) space of all continuous functions on this compact space, endowed with the maximum norm $\|\cdot\|$.
Consider the multivalued mapping
\begin{eqnarray*}%\label{vejce}
{\C(B_{X^*})}\ni f\lll \big\{x^*\in B_{X^*}:\ f(x^*)=\max f(B_{X^*})\big\}=:\partial(f)\, \subset B_{X^*};
\end{eqnarray*}
thus $\partial: {\C(B_{X^*})}\longrightarrow 2^{X^*}$.
Clearly, for every $f\in {\C(B_{X^*})}$ the set $\partial(f)$ is non-empty and weak$^*$ compact. The mapping $\partial$ is also norm-to-weak$^*$ upper semicontinuous. To check this, consider any 
weak$^*$ open set $W$ in $X^*$. We
have to show that $\{f\in\C(B_{X^*}):\ \partial(f)\subset W\}$ is an open set. Assume that this is not the case. Then there exists a sequence $f_0,f_1,f_2,\ldots$ of elements in $\C(B_{X^*})$ such that $\partial(f_0)\subset W$, that $\|f_n - f_0\| \to 0$ as $n\to\infty$, and that for every $n\in\N$ there is $x^*_n\in\partial(f_n)\setminus W$.
Recalling that all the $x^*_n$'s belong to the (weak$^*$ compact) set
$B_{X^*}$, the sequence $(x^*_n)$ has a weak$^*$ cluster point,
$x^*\in B_{X^*}$, say. Also
\begin{eqnarray*}
\max f_0(B_{X^*})&\le& \max f_n(B_{X^*})+\|f_0-f_n\|\\
&=&f_n(x^*_n)-f_0(x^*_n)+\|f_0-f_n\|+f_0(x^*_n)
\le 2\|f_n-f_0\| + f_0(x^*_n)
\end{eqnarray*}
for every $n\in\N$. And, since $f_0$ is weak$^*$ continuous,  $\big(f_0(x^*)\le\big)\ \max f_0(B_{X^*})\le f_0(x^*)$. We got that $x^*\in\partial(f_0)\ (\subset W)$. However, simultaneously, $x^*\not\in W$ as $x^*_n\not\in W$ for all $n\in \N$; a contradiction.

If the Banach space $(X,\|\cdot\|)$ is separable, the mapping $\partial: {\C(B_{X^*})}\longrightarrow 2^{X^*}$ has the following (crucial) property: For every $\xi\in B_{X^*}$ there is an $f\in \C(B_{X^*})$ such that
$\partial(f)=\{\xi\}$. (This is the very point where the approach of Ch. Stegall starts.) 
Indeed, since the compact space $(B_{X^*},w^*)$ is then metrizable, 
$\xi$ is a $G_\delta$ point in it and Urysohn's lemma \cite[Theorem 1.5.10]{e} provides easily a suitable function $f$. More specifically (and without using any topological tool), such an $f$ can be constructed
by the formula 
\begin{eqnarray}\label{zero}
B_{X^*}\ni x^*\longmapsto 1-\sum_{n=1}^\infty 2^{-n}\big|\la \xi,s_n\ra-\la x^*,s_n\ra\big| =:f(x^*)
\end{eqnarray}
 where $\{s_1,s_2,\ldots\}$ is a fixed countable dense subset 
of the closed unit ball $B_X$ in $X$.
(By \cite[Lemma 2.2.1(ii)]{f}  we know that the function
$C(B_{X^*})\ni g\longmapsto \max g(B_{X^*})$ is Gateaux differentiable at
$g:=f$.)

In what follows, let $(X,\|\cdot\|)$ be a Banach space of an arbitrary density.
The observation in the latter paragraph leads to introducing the following family of functions from $\C(B_{X^*})$.
Let $S$ be any nonempty set in $B_X$. By ${\ll}(S)$  we denote the family of all 
functions of the form
\begin{eqnarray*}%\label{17}
B_{X^*}\ni x^*\longmapsto 1-\sum_{n=1}^k 2^{-n}\big|a_n-\la x^*,s_n\ra\big|,
\end{eqnarray*}
where $k\in\N,\ a_1,a_2,\ldots,a_k$ are rational numbers in $[-1,1]$, and $s_1,s_2,\ldots,s_k$ 
are elements from $S$; note that $\#\ll (S)=\#S+\aleph_0$. It is easy to check that each element of ${\ll}(S)$ is weak$^*$ continuous and has maximum norm at most equal to $1$. Thus ${\ll}(S)$ is a subset of (the closed unit ball of) $\,\C(B_{X^*})$. We can easily check that 
\begin{eqnarray}\label{17}
\overline{\ll(\overline{S})} = \overline{\ll  (S)}\,;
\end{eqnarray}
hence, the set in (\ref{17}) is separable whenever $S$ is separable.
Note that the $f$ defined in (\ref{zero}) belongs to the (norm) closure of
$\ll(\{s_1,s_2,\ldots\})\,$. 
%Also, we can easily observe (using the previous paragraph) that whenever $V$ is a separable subspace of $X$ and %$v^*\in B_{V^*}$, then there exists an $f\in \ll(B_V)$ so that $\partial (f)\!\!\upharpoonright_V=\{v^*\}$.

From now on, assume that $(X,\|\cdot\|)$ is an Asplund space, which (equivalently) means that
the dual space $X^*$ is weak$^*$ dentable \cite[Theorem 2.32]{ph}. Then, we are ready to apply the selection theorem of Jayne and Rogers 
\cite[Theorem 8]{jr}, \cite[Theorem 8.1.2]{f} to our multivalued mapping $\partial$ 
(which is already known to be norm-to-weak$^*$ upper
semicontinuous and weak$^*$ compact valued). Thus we get a sequence $\lambda_j: {\C(B_{X^*})}\longrightarrow X^*, \ j\in\N$, of norm-to-norm
continuous mappings such that for every $f\in {\C(B_{X^*})}$ the limit
$\lim_{j\to\infty} \lambda_j(f)=:\lambda_0(f)$ exists in the norm topology
of $X^*$ and moreover $\lambda_0(f)\in\partial(f)$, that is, 
$f\big(\lambda_0(f)\big)=\max f(B_{X^*})$. 
Now, we define the multivalued mapping
$$
{\C(B_{X^*})}\ni f\longmapsto \{\lambda_1(f),\lambda_2(f),\ldots\}
=:\Lambda(f)\subset X^*;
$$
thus $\Lambda: {\C(B_{X^*})}\longrightarrow 2^{X^*}$.
The  continuity of the mappings $\lambda_j$'s and (\ref{17}) then guarantee that
\begin{eqnarray}\label{18}
 \overline{\Lambda\big(\overline{\ll(\overline{S})}\big)}
 =\overline{\Lambda\big(\ll  (S)\big)}\, .
\end{eqnarray}
%\subset \overline{\Lambda\big(\ll  (S)\big)}$.
%Also, we can easily verify that ${\ll}(S)$ is separable whenever $S$ is a countable, or
%even separable, set and $X$ is an Asplund space. 
%hence, the set in (\ref{18}) is a separable subset of $X^*$ whenever $S\subset B_X$ is separable.

\begin{proposition}\label{0}
Let $(X,\|\cdot\|)$ be any Asplund space and let $V$ be any subspace of it. 
Then $B_{V^*} \subset \overline{\Lambda\big({\ll(B_V)}\big)}\,\!\!\upharpoonright_V\,$ and 
$\dens V=\dens\ov{{\rm sp}\,\Lambda\big({\ll(B_V)}\big)}=\dens V^*$;
in particular, $B_{X^*} \subset \overline{\Lambda\big({\ll(B_X)}\big)}$.
\end{proposition}

\begin{proof}
First, assume that $V$ is separable. Consider any $v^*\in B_{V^*}$.
Find a countable dense subset $\{s_1,s_2,\ldots\}$ of $B_V$. Define 
$f:=1-\sum_{n=1}^\infty 2^{-n}\big|\la v^*,s_n\ra -\la\cdot,s_n\ra\big|$;
thus $f\in\overline{\ll(B_V)}$. Then $f(x^*)=1$ for every $x^*\in B_{X^*}$, with
${x^*}\!\!\upharpoonright_{V}=v^*$, and $-1\le\  f(y^*)<1$ whenever
$y^*\in B_{X^*}$ and ${y^*}\!\!\upharpoonright_V\neq v^*$. 
Thus $\partial(f)\!\!\upharpoonright_V=\{v^*\}$, and, since 
$\lambda_0(f)\in\partial(f)$, we have that $\lambda_0(f)\!\!\upharpoonright_V=v^*$. 
Recalling that $\|\lambda_j(f)-\lambda_0(f)\|\longrightarrow0$ as $j\to\infty$, 
we can conclude that 
$$
\lambda_0(f)\in\overline {\Lambda(f)} 
\subset \overline{\Lambda\big(\overline{\ll(B_V)}\,\big)}
=\overline{\Lambda(\ll(B_V))}
$$ 
by (\ref{18}). Therefore, 
$v^*=\lambda_0(f)\!\!\upharpoonright_{V}\in \overline{\Lambda\big(\ll(B_V)\big)}\,\!\!\upharpoonright_V \,$.

Second, assume that $V$ is non-separable (provided that $X$ is non-separable). 
Consider any $v^*\in B_{V^*}$. Find $\xi\in B_{X^*}$ so that $\xi\!\!\upharpoonright_V=v^*$.
We shall construct countable sets $S_0\subset S_1\subset S_2\subset\cdots\subset B_V$ and separable subspaces
$Z_0\subset Z_1\subset Z_2\subset\cdots\subset V$ as follows.
Pick any countable subset $S_0$ of $B_V$ and any separable (rather infinite-dimensional) subspace $Z_0\subset V$.  Let $m\in\N$ be given and assume that we have found $S_{m-1}$ and $Z_{m-1}$.
Clearly, 
${\Lambda\big(\ll(S_{m-1}})\,\big)$ is a countable subset
of $X^*$. Find then a countable set $S_{m-1}\subset S_m\subset B_V$
such that $\overline{S_m}\supset B_X\cap{Z_{m-1}}$ and that
\begin{equation}\label{19}
 \|\xi- {x^*}\| =\sup\big\{\la \xi,s\ra - 
\la x^*,s\ra:\ s\in S_m\big\}\ \ \hbox{for every}\  
x^*\in \Lambda\big(\ll(S_{m-1})\big)\,.
\end{equation}
Put then $Z_m:=\overline{{\rm sp}\, (Z_{m-1}\cup S_m)}$.
Do so for every $m\in\N$ and put finally $S:=S_0\cup S_1\cup S_2\cup\cdots$
and $Z\!:=\overline{Z_0\cup Z_1\cup Z_2\cup\cdots}\,$. 
Clearly, $S$ is a countable set, $Z$ is a separable subspace of $V$,
and $\overline{S}=B_Z$. 
The ``separable'' case says that ${\xi}\!\!\upharpoonright_Z \in
\overline{\Lambda\big(\ll(B_{Z})\big)}\, \!\!\upharpoonright_Z \,$.
By (\ref{18}) we get that $\xi\!\!\upharpoonright_Z$ actually belongs to the set
$\overline{\Lambda(\ll  (S))}\!\!\upharpoonright_Z\,$. Pick $x^*\in \overline{\Lambda(\ll  (S))}$ so that $\xi\!\!\upharpoonright_Z=x^*{}\!\!\upharpoonright_Z$. 
For every $i\in\N$ find $x^*_i\in\Lambda(\ll  (S))$ so that
$\|x^*-x^*_i\|<\frac1i$. A moment reflection reveals that for
every $i\in\N$ there is $m_i\in\N$ so that
$x^*_i\in\Lambda(\ll  (S_{m_i-1}))$. Then, by (\ref{19}),
\begin{eqnarray*}
\|\xi-x^*_i\|&=&\sup\big\{\la\xi,s\ra - \la x^*_i,s\ra:\ s\in S_{m_i}\big\}\\
&\le&\big\|\xi\!\!\upharpoonright_Z-x^*_i\!\!\upharpoonright_Z\!\big\| = \big\|x^*\!\!\upharpoonright_Z-x^*_i\!\!\upharpoonright_Z\big\| 
\le \|x^*-x^*_i\|<\hbox{$\frac1i$}
\end{eqnarray*}
for every $i\in\N$. Therefore 
$\xi \in \overline{\Lambda(\ll  (S))} \subset\overline{\Lambda(\ll(B_V))}\, ,$
and so, $v^*=\xi\!\!\upharpoonright_V\in \overline{\Lambda\big(\ll(B_V)\big)}\,\!\!\upharpoonright_V \,$. 

Finally, consider any subspace $V$ of $X$. Choosing a dense subset $M$ of $B_V$, with $\#M=\dens V$, we have by (\ref{18})
\begin{eqnarray*}
\dens V&\le&\!\! \dens V^* \le \dens \overline{\Lambda\big(\ll(B_V)\big)}\,\!\!\upharpoonright_V 
\le\dens \overline{\Lambda\big(\ll(B_V)\big)}
=\dens \overline{{\rm sp}\Lambda\big(\ll(B_V)\big)}\\
&=&\!\! \dens\ov{\Lambda(\ll(M))}
\le \#\Lambda\big(\ll(M)\big) = \# M\ =\dens V.
\end{eqnarray*}
\end{proof}

\begin{rem}\label{7} 
{\em Proposition \ref{0} is crucial for dropping S. Simons' lemma from the original construction of a projectional resolution of the identity in duals to Asplund spaces; see \cite{fg}
modulo \cite{f1}. Indeed, in \cite{fg}, instead of the mapping
$\partial:\C(B_{X^*})\longrightarrow 2^{B_{X^*}}$, there is considered
the (so called duality) mapping $X\ni x\longmapsto \{x^*\in S_{X^*}\!: \, \langle x^*,x\rangle=\|x\|\}
=:J(x)$. This $J: X \longrightarrow 2^{S_{X^*}}$ is also 
norm-to-weak$^*$ upper semicontinuous and weak$^*$ compact valued. Hence, 
the Jayne-Rogers theorem \cite[Theorem 8.1.2]{f} yields norm-to-norm continuous mappings $D_1, D_2, \ldots$ (now)
from $X$ into $X^*$ such that for every $x\in X$ the limit $\lim_{n\to\infty} D_n(x)=:D_0(x)$
exists in the norm topology of $X^*$ and moreover $\langle D_0(x),x\rangle=\|x\|$. 
Yet, there is no obvious guarantee that
the set $D_0(X)$ is equal to, or at least norm-dense in the unit sphere $S_{X^*}$. 
Here, S. Simons' lemma enters the argument and remedies the
situation by showing that $D_0(X)$ is linearly dense in all of $X^*$.}
\end{rem} 

\begin{proposition}\label{1}
Let $(X,\|\cdot\|)$ be a non-separable Asplund space and let $Z\subset X$ be
an infinite-dimensional subspace with {\rm dens}$\, Z <\, ${\rm dens}$\, X$. Then there exists an overspace 
$Z\subset V\subset X$, with {\rm dens}$\, V=\,${\rm dens}$\, Z$, such that the
restriction mapping $\ \overline{ {\rm sp}\,\Lambda\big(\ll(B_V)\big)}\ni x^*\longmapsto {x^*}\!\!\upharpoonright_V=:R(x^*)\in V^*$ 
is a (surjective) isometry.
\end{proposition}

\begin{proof}
Put $\aleph:=$\, dens$\, Z$.
By induction, we shall construct sets $S_0\subset S_1\subset S_2\subset\cdots \subset B_X$, all of cardinality $\aleph$, as follows. Let $S_0$ be a dense subset of $B_Z$, with $\# S_0=\aleph$. Let $m\in\N$ be fixed and assume that we have already found $S_{m-1}$. We already know that $\#\ll(S_{m-1})=\aleph$.  Hence, by the norm-to-norm continuity of the mappings $\lambda_j$'s, we have that $\,\overline{{\rm sp}\, \Lambda\big(\ll(S_{m-1})\big)}$ is a subspace of $X^*$, with density $\aleph$;
let $M$ be a dense set in it, with $\#M=\aleph$. Find a set $S_{m-1}\subset S_m\subset B_X$,
with $\#S_m=\aleph$, so big that $\,$sp$\,S_{m-1}\cap B_X\subset
\overline{S_m}$ and that $\|\xi\|=\sup\,\la \xi,S_m\ra$ for every $\xi\in M$. 
Then, of course, $\|x^*\|=\sup\,\la x^*,S_m\ra$ for every 
$x^*\in{{\rm sp} \,\Lambda\big(\ll(S_{m-1})\big)}$. This finishes the induction (step). 

Having constructed the $S_m$ for every $m\in\N$, put $S:=S_1\cup S_2\cup\cdots$ and $V:=\overline{{\rm sp}\, S}$;
then $\#S=\aleph$ and $V$ is a subspace of $X$, with density $\aleph$.
We shall show that this $V$, together with the corresponding $R$, serve for the conclusion of our proposition. Consider any $x^*\in
\overline{{\rm sp}\, \Lambda\big(\ll(B_V)\big)}$ and let $\ee>0$ be arbitrary.
It is easy to verify that $\overline S=B_V$. By (\ref{18}),
$\overline{\Lambda(\ll(B_V))}=\overline{\Lambda(\ll  (S))}$,
and hence $\overline{{\rm sp}\,\Lambda(\ll(B_V))}=
\overline{{\rm sp}\,\Lambda(\ll  (S))}\,$.
Thus $x^*$ belongs to $\overline{{\rm sp}\, \Lambda\big(\ll  (S)\big)}$.
Find then $\xi\in {\rm sp}\, \Lambda\big(\ll  (S)\big)$ so that $\|x^*-\xi\|<\ee$. We remark that
$\xi$ belongs even to ${\rm sp}\, \Lambda\big(\ll  (S_{m-1})\big)$
for some big $m\in\N$. Then
\begin{eqnarray*}
\|x^*\|-\ee &<&\|\xi\| = \sup\,\la \xi,S_m\ra \le \sup\,\la \xi,B_V\ra\\
& =& \|R(\xi)\| < \|R(x^*)\|+ \ee \le \|x^*\|+\ee.
\end{eqnarray*}
Thus $\|x^*\|=\|R(x^*)\|$. We proved that $R$ is an isometry. That
$R$ is surjective follows immediately from Proposition~\ref{0}.
\end{proof}

\begin{proposition}\label{2}
Let $V$ be a subspace of an Asplund space $(X,\|\cdot\|)$ such that the
restriction mapping $\ \overline{ {\rm sp}\,\Lambda\big(\ll(B_V)\big)}\ni x^*\longmapsto {x^*}\!\!\upharpoonright_V=:R(x^*)\in V^*$ 
is a (surjective) isometry. Then the mapping $\ X^*\ni x^*\longmapsto
R^{-1}\big({x^*}\!\!\upharpoonright_V\big)=: P(x^*)$ is a linear norm-$1$ projection, $P(X^*)=\overline{{\rm sp}\,\Lambda\big(\ll(B_V)\big)}$, $\dens P(X^*)=\dens V$, and $\overline{V}^{\, w^*}=P^*(X^{**})$.
\end{proposition}

\begin{proof}
The first three statements concerning $P$ immediately follow from
the definition of it. The ``density'' statement is contained in Proposition~\ref {0}. It remains to prove the last equality. That $V\subset P^*(X^{**})$ follows from the definition of $P$; hence 
$\overline{V}^{\,w^*}\subset P^*(X^{**})$. Assume there exists $x^{**}\in P^*(X^{**}) \setminus \overline{V}^{\,w^*}\!$. The Hahn-Banach separation theorem yields an $x^*\in X^*$ such that $\langle x^{**},x^*\rangle\neq0$ and ${x^*}\!\!\upharpoonright_{V}\equiv 0$. But 
$$
\langle x^{**},x^*\rangle=\langle P^*(x^{**}),x^*\rangle=\langle x^{**},P(x^*)\ra = \big\langle x^{**},R^{-1}\big({x^*}\!\!\upharpoonright_{V}\big)\big\rangle =0;
$$ a contradiction. 
\end{proof}

\begin{proposition}\label{3}
Let $V_1,\, V_2$ be two subspaces of an Asplund space $(X,\|\cdot\|)$ such that $V_1\subset V_2$ and that the
restriction mappings $\ \overline{ {\rm sp}\,\Lambda\big(\ll(B_{V_i})\big)}\ni x^*\longmapsto {x^*}\!\!\upharpoonright_{V_i}=:R_i(x^*)\in V_i^*,\ i=1,2$, 
are (surjective) isometries. Define $P_i: X^*\rightarrow X^*$ by
$P_i(x^*)={R_i}^{-1}\big({x^*}\!\!\upharpoonright_{V_i}\big),\ x^*\in X^*,\ 
i=1,2$. Then $P_1\!\circ \!P_2 = P_1\ \,(=P_2\!\circ \!P_1)$, and
$(P_2-P_1)(X^*)$ is isometrical with $(V_2/V_1)^*$.
\end{proposition} 

\begin{proof}
(i) From the definition of $P_i$'s we have immediately that $P_2\circ P_1= P_1$.
Now, consider any $x^*\in X^*$ and any $x^{**}\in X^{**}$. Since $\overline{V_1}^{\,w^*}=P^*_1(X^{**})$ by Proposition~\ref{2}, 
there is a net $(v_\tau)_{\tau\in T}$ in $V_1$ which weak$^*$ converges to  $P^*_1(x^{**})$.
Then, using the definition of $R_2$ and the inclusion $V_1\subset V_2$, we get
\begin{eqnarray*}
\la P^*_2\circ P^*_1(x^{**}), x^*\ra &=& \la P^*_1(x^{**}),P_2(x^*)\ra
=\lim_{\tau\in T}\,\la P_2(x^*),v_\tau\ra \\
&=& \lim_{\tau\in T}\, \la R_2{}^{-1}(x^*\!\!\upharpoonright_{V_2}),v_\tau\ra
=\lim_{\tau\in T}\,\la x^*,v_\tau\ra =\la P^*_1 (x^{**}),x^*\ra,
\end{eqnarray*}
and so $P^*_2\circ P^*_1 = P^*_1$, that is, $P_1\circ P_2=P_1 \,$.
The ``isometrical'' statement can be shown as in the proof of \cite[Proposition 6.1.9(iv)]{f}.
\end{proof}

Now we are armed to construct a projectional resolution of the identity on the dual to every Asplund space.
But we rather prefer to present a bit stronger/richer statement. 

\begin{theorem}\label{main} Let $(X,\|\cdot\|)$ be a non-separable Asplund space.
Then there exist a family $\vv$ of subspaces of $X$ and a family 
$\{Y_V\!:\ V\in\vv\}$ of subspaces of $X^*$ such that
\begin{enumerate}[\upshape (i)]
  \item $\bigcup\big\{V:\ V\in\vv \ {\it and}\ \dens V\!=\!\aleph\big\}=X$
  and $\,\bigcup\big\{Y_V\!: V\in\vv \ {\it and}\ \dens V\!=\!\aleph\big\}$ $=X^*$
  for every infinite cardinal $\aleph<\dens X$;
  \item if $V_1,V_2\in \vv$, there is $V\in\vv$ such that $V\supset V_1\cup V_2$ and 
        $\dens V=\max\{\dens V_1,\dens V_2\}$. 
  \item for every $V\in\vv$ the assignment        
        $Y_V\!\ni\! x^*\longmapsto x^* \!\!\upharpoonright _V =:R_V(x^*)\in                 V^*$ is a surjective isometry, and hence the mapping $X^*\ni x^*                \longmapsto R_V{}^{-1}(x^*\!\!\upharpoonright_V)$$=:P_V(x^*)$ is a                norm-$1$ linear projection on $X^*$,
        with range $Y_V$, and $\dens P_V(X^*)=\dens V$;
  \item $\ov{V}^{\,w^*} = {P_V}^*(X^{**})$ for every $V\in\vv$;
  
  \item $\vv$ is complete in the following sense: if $\gamma$ is a limit ordinal,
        and $\{V_\alpha:\ 1\le\alpha<\gamma\}$ is an increasing long sequence
        of elements of $\,\vv$, then $V:=\ov{\bigcup_{1\le\alpha<\gamma} 
        V_\alpha}$ belongs to $\vv$ and $Y_V=\ov{\bigcup_{1\le\alpha<\gamma}
        Y_{V_\alpha}}\,$;
  \item if $\,V,U\in\vv$ and $V\subset U$, then $Y_V\subset Y_U$,
        $P_V\circ P_U=P_V\ \,(=P_U\circ P_V)$, and  
  %\item if $\,V,U\in\vv$ and $V\subset U$, then 
        $(P_U-P_V)(X^*)$ is isometrical with $(U/V)^*$.                 
\end{enumerate}
\end{theorem}

\begin{proof}
For every subspace $V$ of $X$ we put $Y_V:=\ov{{\rm sp}\,\Lambda(\ll(B_V))}$ 
and we consider the assignment $Y_V\ni x^*\longmapsto x^*\!\!\upharpoonright _V=:R_V(x^*)\in V^*$.
Let $\vv$ consist of all subspaces $V$ of $X$ such that $R_V: Y_V\longrightarrow V^*$ is a surjective isometry.

\noindent (i) and (ii) are guaranteed by Propositions~\ref{1} and \ref{0}.

\noindent(iii) follows from the definition of $\vv$ via Propositions~\ref{0}  and \ref{2}. 

\noindent(iv) follows from Proposition~\ref{2}.

\noindent (v) Assume the premise here holds. (\ref{18}) and some elementary reasoning yields 
\begin{eqnarray*}
\Lambda(\ll(B_V)) &\subset& \ov{\Lambda(\ll(B_V))} 
=\ov{\Lambda\big(\ll(\ov{\hbox{$\bigcup_{\alpha<\gamma}$}B_{V_\alpha}}\big)\big)}
= \ov{\Lambda\big(\ll(\hbox{$\bigcup_{\alpha<\gamma}$}B_{V_\alpha}\big)\big)}\\
&=& \ov{\hbox{$\bigcup_{\alpha<\gamma}$}\Lambda\big(\ll(B_{V_\alpha})\big)}
\subset \ov{\hbox{$\bigcup_{\alpha<\gamma}$}\ov{{\rm sp}\Lambda\big(\ll(B_{V_\alpha})\big)}}
=\ov{\hbox{$\bigcup_{\alpha<\gamma}$} Y_{V_\alpha}}.
\end{eqnarray*}
Hence
\begin{eqnarray}\label{pet}
Y_V = \ov{{\rm sp}\Lambda\big(\ll(B_V)\big)} = \ov{\hbox{$\bigcup_{\alpha<\gamma}$} Y_{V_\alpha}}.
\end{eqnarray}
It remains to show that $V\in\vv$, that is, that the mapping $R_V: Y_V\longrightarrow V^*$
is a surjective isometry. By Proposition~\ref{0} and (\ref{pet}), $R_V$ is surjective.
Further, fix any $x^*\in Y_V$ and let $\ee>0$ be arbitrary.
Find $\xi\in \bigcup_{\alpha<\gamma}Y_\alpha$
such that $\|x^*-\xi\|<\ee$. Then $\xi$ belongs to $Y_\alpha$ for some $\alpha<\gamma$.
Now, as $V_\alpha\in\vv$, we have
\begin{eqnarray*}
\|x^*\|-\ee &<& \|\xi\| = 
\|R_{V_\alpha}(\xi)\| =\big\|\xi\!\!\upharpoonright_{V_\alpha}\big\|
\le \big\|\xi\!\!\upharpoonright_{V}\big\|\\ &<& \big\|x^*\!\!\upharpoonright_{V}\big\|
+\ee = \|R_V(x^*)\| + \ee \le\|x^*\| + \ee.
\end{eqnarray*}
Therefore, $\|x^*\|=\|R_V(x^*)\|$ and the mapping $R_V$ is
shown to be an isometry. Thus $V\in\vv$.

\noindent (vi) immediately follows from Proposition~\ref{3}.

\end{proof} 

\begin{corollary}\label{hlavni}
Let $(X,\|\cdot\|)$ be a non-separable Asplund space. Then $(X^*,\|\cdot\|)$ admits 

\noindent (i) {\rm \cite{fg}} a projectional resolution of the identity, and also 

\noindent (ii) {\rm \cite{kubis}} a $1$-projectional skeleton. 
\end{corollary}

\begin{proof} (i) Let $\mu$ be the first ordinal with $\#\mu=\dens X$. Find
a dense subset $\{x_\alpha:\ \omega<\alpha<\mu,\,$ $\alpha$ is non-limit$\}$
in $X$. By (i) find a separable element $V_\omega\in\vv$. Let $\gamma\in
(\omega,\mu]$ be any ordinal an assume that we already found $V_\alpha\in\vv$
for every $\omega\le\alpha<\gamma$.
Assume first that $\gamma$ is non-limit. By (i) find a $V\in\vv$ such that $V\ni x_\gamma$ and $\dens V=\dens V_{\gamma-1}$. By (ii) find a 
$V_\gamma\in\vv$ such that $V_\gamma\supset V\cup V_{\gamma-1}$ and
$\dens V_\gamma=V_{\gamma-1}$. Second, assume that $\gamma$ is
a limit ordinal. Put then  $V_\gamma=\ov{\bigcup_{\omega<\alpha<\gamma} V_\alpha}$. By (v), we have that $V_\gamma\in\vv$ and $Y_{V_\gamma}=\ov{\bigcup_{\omega\le\alpha<\gamma}
        Y_{{V_\alpha}}}\,$. Now we can immediately verify that
$\big\{ P_{V_\alpha}:\ \omega\le\alpha\le\mu\big\}$ is a 
projectional resolution of the identity on $(X^*,\|\cdot\|)$; see, e.g. \cite[Definition 6.1.5]{f}. 

(ii) We note that the relation ``$\subset$'' on the family $\vv$ from Theorem~\ref{main} 
is a directed partial order, which is moreover $\sigma$-complete. Thus, the subfamily 
$\{P_V:\ V\in\vv\ {\rm and}\ \dens V=\aleph_0\}\,$ is a  1-{projectional skeleton} on $(X^*,\|\cdot\|)$ 
in the sense of the definition in \cite[pages 369, 370]{kkl}.
\end{proof}

\begin{rem} {\em

1.  A 1-projectional skeleton in the dual to an Asplund space was originally constructed by W. Kubi\'s \cite{kubis}.
He started from the existence of the so called {\em projectional generator} in the dual to an Asplund space
\cite[Proposition 8.2.1]{f} and then he proceeded using elementary submodels from logic. It should be noted that the proof of \cite[Proposition 8.2.1]{f} is based on S. Simons' lemma and on the Jayne-Rogers selection theorem \cite[Theorem 8]{jr}.

2. Let $\nn_1,\ \nn_2,\ \ldots$ be a sequence of equivalent norms on
an Asplund space $X$. Let $\vv_1,\ \vv_2,\ \ldots$ be
corresponding families found in Theorem~\ref{main} for these norms.
It is not difficult to show that the family
$\vv:=\bigcap_{i=1}^\infty\vv_i$ (is not only non-empty but that it even) satisfies all the conditions (i) -- (vi)
of Theorem ~\ref{main}; see the proof of \cite[Proposition 1.1]{bm}. Then Corollary~\ref{hlavni}
provides one projectional resolution of the identity and one $1$-projectional skeleton on $X^*$
which can be related to any of the norms $\nn_1,\ \nn_2,\ \ldots$

3. Theorem~\ref{main} can be proved also with help of tools from \cite{fg} (where S. Simons' lemma
was used). Indeed, let the mappings $D_n: X\rightarrow X^*,\ n\in\N$, be from Remark~\ref{7} (they come from the Jayne-Rogers theorem) and define $D: X\rightarrow 2^{X^*}$ by
$D(x)= \{D_1(x),D_2(x),\ldots\},\ x\in X$. 
For every subspace $W$ of $X$ put $Y_W:=\ov{{\rm sp}\,D(W)}$ 
and  consider the assignment $Y_W\ni x^*\longmapsto x^*\!\!\upharpoonright _W=:R_W(x^*)\in W^*$.
Let $\mathcal W$ consist of all subspaces $W$ of $X$ such that $R_W: Y_W\longrightarrow W^*$
is a surjective isometry. From \cite[pages 145, 146]{fg}, where S. Simons' lemma and other things were used, we know that
\begin{eqnarray}\label{fg}
\ov{{\rm sp}\, D(W)\!\!\upharpoonright_W}=W^*\quad \text{for every subspace}\quad W\subset X.
\end{eqnarray} 
(This is a weakened analogue of our
Proposition~\ref{0}). By \cite[Lemma 1]{f1}, for every $Z\subset X$, with
$\dens Z<\dens X$ there is a subspace $Z\subset W\subset X$, with
$\dens W=\dens Z$, such that $R_W$ is an 
isometry from $Y_W$ into $W^*$. (This is an analogue of our Proposition~\ref{1}.)
Using this, we easily get that 
$\ov{{\rm sp}\, D(W)\!\!\upharpoonright_W}=\ov{{\rm sp}\, D(W)}\!\!\upharpoonright_W$, and by (\ref{fg}),
$R_W: Y_W\rightarrow W^*$ is a surjective isometry. Thus $W\in\mathcal W$. This way we get the properties (i) -- (iv) and (vi) listed in Theorem~\ref{main}. 
(v) can be proved as in Theorem~\ref{main}, now from the norm-to-norm continuity of the $D_n$'s.  

4. We use an opportunity to fix some inaccuracy in the book \cite{f}.
In the proof of \cite[Proposition 8.2.1]{f}, on the page 152, the equality (2) should read as
$\overline{{\rm sp}\,\Phi(B_0)\!\!\upharpoonright_Y}=Y^*$
and the set $\Delta$ should be defined as $\overline{\Phi(B_0)\!\!\upharpoonright_Y}\cap B_{Y^*}\,$.}
\end{rem}

\bf Acknowledgment. \rm
We thank our colleague W. Kubi\'s for discussing the topic of this note.

{}

\end{document}